\documentclass[reqno0,fleqn,12pt]{amsart}
\pdfoutput=1
\usepackage{dsfont}
\usepackage{bbm}
\usepackage[nodate]{datetime}
\usepackage[hmargin=30mm,top=30mm,bottom=30mm,a4paper]{geometry}
\usepackage{color}
\usepackage{subfig}
\usepackage{fancyhdr}
\usepackage{amsmath,amssymb,amsfonts,amsthm}
\usepackage{amstext}
\usepackage{amsmath}
\usepackage{amssymb}
\usepackage{enumerate}
\usepackage{amsbsy}
\usepackage{amsopn}
\usepackage{bbm,amsthm}
\usepackage{amscd}
\usepackage{amsxtra}

\newtheorem{theorem}{Theorem}[section]
\newtheorem*{theorem*}{Theorem}
\newtheorem{lemma}{Lemma}[section]

\theoremstyle{remark}

\newcommand{\CC}{\mathds{C}}

\newcommand{\RR}{\mathds{R}}
\newcommand{\NN}{\mathds{N}}
\newcommand{\HH}{\mathds{H}}

\newcommand{\ind}{\mathds{1}}

\newcommand{\EE}{\mathbb{E}}
\newcommand{\PP}{\mathbb{P}}

\newcommand{\aoe}{\bigg{(}}
\newcommand{\aod}{\bigg{)}}

\DeclareMathOperator{\var}{Var}

\begin{document}
\title{Real zeros of Random Dirichlet series}
\author{Marco Aymone}
\maketitle
\begin{abstract}
Let $F(\sigma)$ be the random Dirichlet series $F(\sigma)=\sum_{p\in\mathcal{P}} \frac{X_p}{p^\sigma}$, where $\mathcal{P}$ is an increasing sequence of
positive real numbers and $(X_p)_{p\in\mathcal{P}}$ is a sequence of i.i.d. random variables
with $\PP(X_1=1)=\PP(X_1=-1)=1/2$. We prove that, for certain conditions on $\mathcal{P}$, if $\sum_{p\in\mathcal{P}}\frac{1}{p}<\infty$ then with positive probability
$F(\sigma)$ has no real zeros while if $\sum_{p\in\mathcal{P}}\frac{1}{p}=\infty$, almost surely $F(\sigma)$ has an infinite number of real zeros.
\end{abstract}

\section{Introduction.}
A Dirichlet series is an infinite sum of the form $F(\sigma):=\sum_{p\in\mathcal{P}}\frac{X_p}{p^{\sigma}}$, where $\mathcal{P}$ is an
increasing sequence of positive real numbers and $(X_p)_{p\in\mathcal{P}}$ is any sequence of complex numbers. If $F(\sigma)$ converges then $F(s)$ converges for all
$s\in\CC$ with real part greater than $\sigma$ (see \cite{montgomerylivro} Theorem 1.1). The abscissa of convergence of a Dirichlet series is the smallest number $\sigma_c$
for which $F(\sigma)$ converges for all $\sigma>\sigma_c$.

The problem of finding the zeros of a Dirichlet series is classical in Analytic Number Theory. For instance, the Riemann hypothesis states that the zeros of the analytic continuation
of the Riemann zeta function $\zeta(\sigma):=\sum_{k=1}^\infty\frac{1}{k^\sigma}$ in the half plane $\{\sigma+it\in\CC:\sigma>0\}$ all have real part equal to $1/2$.
This analytic continuation can be described in terms of a convergent Dirichlet series -- The Dirichlet $\eta$-function
$\eta(s)=\sum_{k=1}^\infty\frac{(-1)^{k+1}}{k^s}$ satisfies $\eta(s)=(1-2^{1-s})\zeta(s)$, for all complex $s$ with positive real part. Thus, to find zeros of $\eta(s)$ for $0<Re(s)<1$ is the same as finding non-trivial zeros of $\zeta$.

In this paper we are interested in the real zeros of the random Dirichlet series $F(\sigma):=\sum_{p\in\mathcal{P}} \frac{X_p}{p^{\sigma}}$, where
the coefficients $(X_p)_{p\in\mathcal{P}}$ are random and $\mathcal{P}$ satisfies:
\begin{align*}
&(P1)\quad \mathcal{P}\cap[0,1)=\varnothing,\\
&(P2) \quad \sum_{p\in\mathcal{P}}\frac{1}{p^\sigma} \mbox{ has abcissa of convergence }\sigma_c=1.
\end{align*}

For instance, $\mathcal{P}$ can be the set of the natural numbers. The conditions $(P1-P2)$ imply, in particular, that the series $\sum_{p\in\mathcal{P}}\frac{1}{p^{2\sigma}}$ converges for each $\sigma>1/2$.
Therefore, if $(X_p)_{p\in\mathcal{P}}$ is a sequence of i.i.d. random variables with $\EE X_p =0 $ and $\EE X_p^2=1$, then, by the Kolmogorov one-series Theorem,
the series $F(\sigma)=\sum_{p\in\mathcal{P}} \frac{X_p}{p^\sigma}$ has \textit{a.s.} abscissa of convergence $\sigma_c=1/2$. Moreover, the function of one complex variable
$\sigma+it\mapsto F(\sigma+it)$ is \textit{a.s.} an analytic function in the half plane $\{\sigma+it\in\CC:\sigma>1/2\}$. In the case $X_p=\pm 1$ with equal probability,
the line $\sigma=\sigma_c$ is a natural boundary for $F(\sigma+it)$, see \cite{kahane} (pg. 44 Theorem 4).

Our main result states:

\begin{theorem}\label{Teorema 1} Assume that $\mathcal{P}$ satisfies $P1$-$P2$ and let $(X_p)_{p\in\mathcal{P}}$ be i.i.d and such that $\PP(X_p=1)=\PP(X_p=-1)=1/2$. Let
$F(\sigma)=\sum_{p\in\mathcal{P}} \frac{X_p}{p^\sigma}$.\\
i. If $\sum_{p\in\mathcal{P}} \frac{1}{p}<\infty$, then with positive probability $F$ has no real zeros;\\
ii. If $\sum_{p\in\mathcal{P}} \frac{1}{p}=\infty$, then \textit{a.s.} $F$ has an infinite number of real zeros.
\end{theorem}

It follows as corollary to the proof of item i. that in the case $\sum_{p\in\mathcal{P}}\frac{1}{p}=\infty$, with positive probability $F(\sigma)$ has no zeros in the interval
$[1/2+\delta,\infty)$, for fixed $\delta>0$.

Since a Dirichlet series $F(s)=\sum_{p\in\mathcal{P}}\frac{X_p}{p^s}$ is a random analytic function, it can be viewed as a random Taylor series
$\sum_{k=0}^\infty Y_k (s-a)^k$, where $a>\sigma_c$ and $(Y_k)_{k\in\NN}$ are random and \textit{dependent} random variables. The case of random Taylor series
and random polynomials where $(Y_k)_{k\in\NN}$ are i.i.d. has been widely studied in the literature, for an historical background we refer to \cite{krishnapurzeros} and
\cite{Vurandompolynomials} and the references therein.

\section{Preliminaries}
\subsection{Notation.} We employ both $f(x)=O(g(x))$ and Vinogradov's $f(x)\ll g(x)$ to mean that there exists a constant $c>0$
such that $|f(x)|\leq c |g(x)|$ for all sufficiently large $x$, or when $x$ is sufficiently close to a certain real number $y$.
For $\sigma\in\RR$, $\HH_{\sigma}$ denotes the half plane $\{z\in\CC:Re(z)>\sigma\}$. The indicator function of a set $S$ is denoted by $\ind_S(s)$
and it is equal to $1$ if $s\in S$, or equal to $0$ otherwise. We let $\pi(x)$ to denote the counting function of $\mathcal{P}$:
\begin{equation*}
\pi(x):=|\{p\leq x:p\in\mathcal{P}\}|.
\end{equation*}

\subsection{The Mellin transform for Dirichlet series} In what follows $\mathcal{P}=\{p_1<p_2<...\}$ is a set of non-negative real numbers satisfying $P1$-$P2$ above.
A generic element of $\mathcal{P}$ is de noted by $p$, and we employ $\sum_{p\leq x}$ to denote $\sum_{p\in\mathcal{P}; p\leq x}$.
Let $A(x)=\sum_{p \leq x} X_p$ and $F(s)=\sum_{p\in\mathcal{P}}\frac{X_p}{p^s}$. Let $\sigma_c>0$ be the abscissa of convergence of $F(\sigma)$. Then $F$ can be
represented as the Mellin transform of the function $A(x)$ (see, for instance, Theorem 1.3 of \cite{montgomerylivro}):
\begin{equation}\label{equation integral representation for Dirichlet series}
F(s)=s\int_{1}^\infty A(x)\frac{dx}{x^{1+s}},\mbox{ for all } s\in\HH_{\sigma_c}.
\end{equation}
In particular, we can state:
\begin{lemma}\label{lemma mellin transform application} Let $F(s)=\sum_{p\in\mathcal{P}} \frac{X_p}{p^s}$ be such that $F(1/2)$ is convergent.
Then for each $\sigma\geq1/2$ and all $\epsilon>0$, for all $U>1$:
\begin{equation*}
F(\sigma+\epsilon)=\sum_{p\leq U} \frac{X_p}{p^{\sigma+\epsilon}} + O\bigg{(}U^{-\epsilon}\sup_{x>U}\bigg{|}\sum_{U< p\leq x}\frac{X_p}{p^\sigma} \bigg{|} \bigg{)},
\end{equation*}
where the implied constant in the $O(\cdot)$ term above can be taken to be $1$.
\end{lemma}
\begin{proof}
Put $\displaystyle A(x)=\sum_{p\leq x} \ind_{(U,\infty)}(p)\frac{X_p}{p^\sigma}$. By (\ref{equation integral representation for Dirichlet series}) it follows that
\begin{align*}
\sum_{p>U} \frac{X_pp^{-\sigma}}{p^{\epsilon}}&=\epsilon\int_{1}^\infty A(x)\frac{dx}{x^{1+\epsilon}}=\epsilon\int_{U}^\infty\bigg{(} \sum_{U<n\leq x} \frac{X_p}{p^\sigma}
\bigg{)}\frac{dx}{x^{1+\epsilon}}\\
&\ll \sup_{x>U}\bigg{|}\sum_{U< p\leq x}\frac{X_p}{p^\sigma} \bigg{|} \int_{U}^\infty \frac{\epsilon}{x^{1+\epsilon}}dx =
U^{-\epsilon}\sup_{x>U}\bigg{|}\sum_{U< p\leq x}\frac{X_p}{p^\sigma} \bigg{|}.
\end{align*}
\end{proof}
\subsection{A few facts about sums of independent random variables} In what follows we use

\textit{Levy's maximal inequality}: Let $X_1,...,X_n$ be independent random variables. Then
\begin{equation}\label{equation Levy maximal}
\PP\bigg{(}\max_{1\leq m \leq n}\bigg{|}\sum_{k=1}^m X_k \bigg{|}\geq t \bigg{)}\leq 3 \max_{1\leq m \leq n}\PP\bigg{(}\bigg{|}\sum_{k=1}^m X_k \bigg{|}\geq \frac{t}{3} \bigg{)}.
\end{equation}

\textit{Hoeffding's inequality}: Let $X_1,...,X_n$ be i.i.d. with $\PP(X_1=1)=\PP(X_1=-1)=1/2$. Let $a_1,...,a_n$ be real numbers. Then for any $\lambda>0$
\begin{equation}\label{equation Hoeffding inequality}
\PP\bigg{(}\sum_{k=1}^n a_k X_k \geq \lambda  \bigg{)}\leq \exp\bigg{(}-\frac{\lambda^2}{2\sum_{k=1}^n a_k^2}  \bigg{)}.
\end{equation}

\section{Proof of the main result}
\begin{proof}[Proof of item i] Since $\sum_{p\in\mathcal{P}}\frac{1}{p}<\infty$ we have by the Kolmogorov one series theorem that the series
$\sum_{p\in\mathcal{P}}\frac{X_p}{\sqrt{p}}$ converges almost surely. In what follows $U>0$ is a large fixed number to be chosen later, $A_U$ is the
event in which $X_p=1$ for all $p\leq U$ and $B_U$ is the event in which
\begin{equation*}
\sup_{x>U} \bigg{|}\sum_{U<p\leq x} \frac{X_p}{\sqrt{p}} \bigg{|} < \frac{1}{10}.
\end{equation*}
 We claim that for sufficiently large $U$ on the event $A_U\cap B_U$ the function
$F(s)=\sum_{p\in\mathcal{P}}\frac{X_p}{p^s}$ does not vanish for all
$s\geq \frac{1}{2}$. Further for sufficiently large $U$ we will show that $\PP(A_U\cap B_U)>0$.

On the event $A_U\cap B_U$ we have by lemma \ref{lemma mellin transform application}
that
\begin{equation}\label{equacao desigualdade sigma + epsilon}
F(1/2+\epsilon)\geq \sum_{p\leq U}\frac{1}{p^{1/2+\epsilon}}-\frac{1}{10U^\epsilon}\geq \frac{\pi(U)}{U^{1/2+\epsilon}}-\frac{1}{10U^\epsilon},
\end{equation}
where $\pi(U)=\#\{p\leq U: p\in\mathcal{P}\}$. We claim that for each $\delta>0$ we have that
\begin{equation*}
\limsup_{U\to\infty}\frac{\pi(U)}{U^{1-\delta}}=\infty.
\end{equation*}
In fact, this is a consequence from P2: For any $\delta>0$ the series diverges $\sum_{p\in\mathcal{P}}\frac{1}{p^{1-\delta}}=\infty$.
To show that this is true we argue by contraposition: Assume that for some fixed $\delta>0$
$\limsup_{U\to\infty}\frac{\pi(U)}{U^{1-\delta}}<\infty$ and hence that there exists a constant $c>0$ such that for all $U>0$,
$\pi(U)\leq c U^{1-\delta}$. In that case we have for $0<\epsilon<\delta$
\begin{align*}
\sum_{p\leq U}\frac{1}{p^{1-\epsilon}}&=\int_1^U \frac{d\pi(x)}{x^{1-\epsilon}}=\frac{\pi(U)}{U^{1-\epsilon}}-\pi(1) +(1-\epsilon)\int_1^U \frac{\pi(x)}{x^{2-\epsilon}} dx \\
& \leq \frac{cU^{1-\delta}}{U^{1-\epsilon}}+1+(1-\epsilon)\int_1^U \frac{cx^{1-\delta}}{x^{2-\epsilon}}dx\ll 1+\int_1^U \frac{1}{x^{1+(\delta-\epsilon)}}dx\ll 1,
\end{align*}
and hence that the series $\sum_{p\in\mathcal{P}}\frac{1}{p^{1-\epsilon}}$ converges. Therefore, we showed that $\limsup_{U\to\infty}\frac{\pi(U)}{U^{1-\delta}}<\infty$
implies that $\sum_{p\in\mathcal{P}}\frac{1}{p^\sigma}$ has abscissa of convergence $\sigma_c\leq 1-\delta$.

Now we may select arbitrarily large values of $U>1$ for which $\pi(U)\geq  U^{1-1/4}$ and $\sum_{p\leq U}\frac{1}{\sqrt{p}}>\frac{1}{10}$, and hence,
by (\ref{equacao desigualdade sigma + epsilon}),
for all $\epsilon>0$ we obtain that
\begin{equation*}
 F(1/2+\epsilon)\geq \frac{ U^{1-1/4} }{U^{1/2+\epsilon}}-\frac{1}{10U^\epsilon}=\frac{1}{U^\epsilon}\bigg{(}U^{1/4}-\frac{1}{10}\bigg{)}>0.
\end{equation*}
This proves that on the event $A_U\cap B_U$ we have that $F(s)\neq 0$ for all $s\in[1/2,\infty)$.

Observe that $A_U$ and $B_U$ are independent and $A_U$ has probability $\frac{1}{2^{\pi(U)}}>0$. Now we will show that the complementary event $B_U^c$ has small probability.
Indeed, by applying the Levy's maximal inequality and the Hoeffding's inequality, we obtain:
\begin{align*}
\PP(B_U^c)&=\lim_{n\to\infty}\PP\bigg{(} \max_{U<x\leq n} \bigg{|}\sum_{U<p\leq x} \frac{X_p}{\sqrt{p}} \bigg{|} \geq  \frac{1}{10}  \bigg{)}\leq 3\lim_{n\to\infty}
\max_{U<x\leq n}\PP\bigg{(} \bigg{|}\sum_{U<p\leq x} \frac{X_p}{\sqrt{p}} \bigg{|} \geq  \frac{1}{30}  \bigg{)}\\
& \leq 6 \lim_{n\to\infty} \max_{U<x\leq n}
\PP\bigg{(} \sum_{U<p\leq x} \frac{X_p}{\sqrt{p}} \geq  \frac{1}{30}  \bigg{)}
\leq  6 \lim_{n\to\infty} \exp\bigg{(}\frac{-1/30^2}{2\sum_{U<p\leq n}\frac{1}{p}}  \bigg{)}\\
& \leq 6 \exp\bigg{(}-\frac{1}{2\cdot30^2\sum_{p>U}\frac{1}{p}}  \bigg{)}.
\end{align*}
Since $\sum_{p\in\mathcal{P}}\frac{1}{p}$ is convergent, the tail $\sum_{p>U}\frac{1}{p}$
converges to $0$ as $U\to\infty$. Therefore, for sufficiently large $U$ we can make $\PP(B_U^c)<1/2$. \end{proof}
Now we are going to prove Theorem \ref{Teorema 1} part $ii$. We present two different proofs. In the first proof we assume that the counting function of $\mathcal{P}$
\begin{equation}\label{equacao pi(x) ll log x}
\pi(x)\ll \frac{x}{\log x}.
\end{equation}
In this case, for instance, $\mathcal{P}$ can be the set of prime numbers. In this proof we show that, for $\sigma$ close to $1/2$, the infinite sum $\sum_{p\in\mathcal{P}}\frac{X_p}{p^\sigma}$ can be approximated
by the partial sum $\sum_{p\leq y}\frac{X_p}{\sqrt{p}}$ for a suitable choice of $y$ (Lemma \ref{lemma aproximando a serie em sigma pela serie em meio }). Then we show that these partial sums change sign for an infinite number of $y$, and hence, $F(\sigma)=\sum_{p\in\mathcal{P}}\frac{X_p}{p^\sigma}$ changes sign for an infinite number of $\sigma\to 1/2^+$.

The case in which $\mathcal{P}$ is the set of natural numbers, the infinite sum $\sum_{p\in\mathcal{P}}\frac{X_p}{p^\sigma}$ can not be approximated by the finite sum $\sum_{p\leq y}\frac{X_p}{\sqrt{p}}$, \textit{i.e}, Lemma \ref{lemma aproximando a serie em sigma pela serie em meio } fails in this case. Thus, our approach is different in the general case. First we show (Lemma \ref{lemma central do limite}) that $\sum_{p\in\mathcal{P}}\frac{1}{p}=\infty$ implies that
\begin{equation}\label{equacao central do limite}
\frac{\sum_{p\in\mathcal{P}}\frac{X_p}{p^\sigma}}{\sqrt{\sum_{p\in\mathcal{P}}\frac{1}{p^{2\sigma}}}}\to_d \mathcal{N}(0,1), \mbox{ as } \sigma\to\frac{1}{2}^+,
\end{equation}
and second, for each $L>0$, the event
\begin{equation*}
\limsup_{\sigma\to\frac{1}{2}^+}\frac{\sum_{p\in\mathcal{P}}\frac{X_p}{p^\sigma}}{\sqrt{\sum_{p\in\mathcal{P}}\frac{1}{p^{2\sigma}}}}\geq L
\end{equation*}
is a tail event, and by (\ref{equacao central do limite}), it has positive probability. Similarly,
\begin{equation*}
\liminf_{\sigma\to\frac{1}{2}^+}\frac{\sum_{p\in\mathcal{P}}\frac{X_p}{p^\sigma}}{\sqrt{\sum_{p\in\mathcal{P}}\frac{1}{p^{2\sigma}}}}\leq -L
\end{equation*}
also is a tail event and has positive probability. Thus, by the Kolmogorov $0-1$ Law, with probability $1$,
$\sum_{p\in\mathcal{P}}\frac{X_p}{p^\sigma}$ changes sign for an infinite number of $\sigma\to 1/2^+$.
\subsection{Proof of Theorem \ref{Teorema 1} (ii) in the case $\pi(x)\ll \frac{x}{\log x}$}

\begin{lemma}\label{lemma aproximando a serie em sigma pela serie em meio } Assume that $\mathcal{P}$ satisfies $P1$-$P2$ and that $\sum_{p\in\mathcal{P}}\frac{1}{p}=\infty$. Further, assume that $\pi(x)\ll\frac{x}{\log x}$. Let $\sigma>1/2$ and $y=\exp((2\sigma-1)^{-1})\geq 10$. Then there is a constant $d>0$ such that for all $\lambda>0$
\begin{equation*}
\PP\bigg{(}\bigg{|} \sum_{p\in\mathcal{P}} \frac{X_p}{p^\sigma}-\sum_{p\leq y} \frac{X_p}{\sqrt{p}} \bigg{|} \geq 2\lambda  \bigg{)}\leq 4\exp(-d\lambda^2).
\end{equation*}
\end{lemma}
\begin{proof} If $|a+b|\geq 2\lambda$ then either $|a|\geq\lambda$ or $|b|\geq \lambda$. This fact combined with the Hoeffding's inequality allows us to bound:
\begin{align*}
\PP\aoe \bigg{|} \sum_{p\in\mathcal{P}} \frac{X_p}{p^\sigma}-\sum_{p\leq y} \frac{X_p}{\sqrt{p}} \bigg{|} \geq 2\lambda\aod \leq & \PP \aoe\bigg{|}
\sum_{p\leq y}X_p \aoe\frac{1}{p^\sigma}-\frac{1}{\sqrt{p}}\aod \bigg{|} \geq \lambda \aod\\+& \PP \aoe\bigg{|}
\sum_{p>y} \frac{X_p}{p^\sigma}\bigg{|} \geq \lambda \aod \\
&\leq \exp\aoe-\frac{\lambda^2}{2V_y}\aod+\exp\aoe-\frac{\lambda^2}{2W_y}\aod,
\end{align*}
where $V_y=\sum_{p\leq y}\aoe \frac{1}{p^\sigma}-\frac{1}{\sqrt{p}} \aod^2$ and $U_y=\sum_{p>y}\frac{1}{p^{2\sigma}}$. To complete the proof we only need to estimate these quantities. By the mean value theorem
\begin{equation*}
\frac{1}{p^\sigma}-\frac{1}{\sqrt{p}}=(\sigma-1/2)\frac{\log p}{p^\theta}, \mbox{ for some }\theta=\theta(p,\sigma)\in [1/2,\sigma].
\end{equation*}
 Therefore
\begin{align*}
V_y&\leq (\sigma-1/2)^2\sum_{p\leq y} \frac{\log^2 p}{p}=(\sigma-1/2)^2\int_{1^-}^y \frac{\log^2 t}{t} d\pi(t)\\
&=(\sigma-1/2)^2\aoe \frac{\pi(y)\log^2 y}{y}-\int_{1^-}^y\pi(t)\frac{2\log t-\log^2t}{t^2}dt\aod\\
&\ll (\sigma-1/2)^2\aoe \log y+\int_{1^{-}}^y\frac{\log t}{t}dt\aod \ll (\sigma-1/2)^2\log^2 y.\\
U_y &=\int_{y}^{\infty}\frac{d\pi(t)}{t^{2\sigma}}=-\frac{\pi(y)}{y^{2\sigma}}-\int_y^{\infty}\frac{-2\sigma\pi(t)}{t^{2\sigma+1}}dt\\
&\ll \frac{1}{y^{2\sigma-1}\log y}+2\sigma\int_{y}^\infty \frac{1}{t^{2\sigma}\log t}dt\ll\frac{1}{y^{2\sigma-1}\log y}+\frac{2\sigma}{(2\sigma-1)y^{2\sigma-1}\log y}\\
&\ll\frac{1}{(2\sigma-1)y^{2\sigma-1}\log y}.
\end{align*}
In particular, the choice $y=\exp ((2\sigma-1)^{-1})$ implies that both variances $V_y$ and $U_y$ are $O(1)$. \end{proof}
The simple random walk $S_n=\sum_{k=1}^n X_n$ where $(X_n)_{n\in\NN}$ is i.i.d with $X_1=\pm 1$ with probability $1/2$ each, satisfies \textit{a.s.}
$\limsup_{n\to\infty} S_n=\infty$ and $\liminf_{n\to\infty} S_n=-\infty$. We follow the same line of reasoning as in the proof of this result
(\cite{shiryaev} pg. 381, Theorem 2) to prove:

\begin{lemma}\label{lemma lim sup acima da variancia} Assume that $\sum_{p\in\mathcal{P}}\frac{1}{p}=\infty$.
Let $y_k$ be a increasing sequence of positive real numbers such that $\lim y_k=\infty$. Then it \textit{a.s.} holds that:
\begin{align*}
&\limsup_{k\to\infty} \frac{\sum_{p\leq y_k} \frac{X_p}{\sqrt{p}}}{\sqrt{\sum_{p\leq y_k} \frac{1}{p}}}=\infty,\\
&\liminf_{k\to\infty} \frac{\sum_{p\leq y_k} \frac{X_p}{\sqrt{p}}}{\sqrt{\sum_{p\leq y_k} \frac{1}{p}}}=-\infty.
\end{align*}
\end{lemma}

\begin{proof}
We begin by observing that $(X_p/\sqrt{p})_{p\in\mathcal{P}}$ is a sequence of independent and symmetric random variables that are uniformly bounded by $1$.
It follows that
\begin{equation*}
\lim_{y\to\infty} \var \sum_{p\leq y} \frac{X_p}{\sqrt{p}}=\lim_{y\to\infty}\sum_{p\leq y}\frac{1}{p}=\infty,
\end{equation*}
and hence this sequence satisfies the Lindenberg condition. By the Central Limit Theorem it follows that for each fixed $L>0$ there exists a $\delta>0$
such that for sufficiently large $y>0$
\begin{equation*}
\PP\aoe  \sum_{p\leq y} \frac{X_p}{\sqrt{p}} \geq L \sqrt{\sum_{p\leq y} \frac{1}{p}}  \aod=\PP\aoe  \sum_{p\leq y} \frac{X_p}{\sqrt{p}} \leq - L \sqrt{\sum_{p\leq y} \frac{1}{p}}  \aod\geq \delta.
\end{equation*}
Next observe that the event in which $\limsup_{k\to\infty} \frac{\sum_{p\leq y_k} \frac{X_p}{\sqrt{p}}}{\sqrt{\sum_{p\leq y_k} \frac{1}{p}}}\geq L$ is a tail event,
and hence by the Kolmogorov zero or one law it has either probability zero or one. Since
\begin{align*}
&\PP \aoe \sum_{p\leq y_k} \frac{X_p}{\sqrt{p}} \geq L \sqrt{\sum_{p\leq y_k} \frac{1}{p}}\mbox{ for infinitely many }k\aod\\
&=\lim_{n\to\infty}\PP \aoe \bigcup_{k=n}^\infty\bigg{[} \sum_{p\leq y_k} \frac{X_p}{\sqrt{p}} \geq L \sqrt{\sum_{p\leq y_k} \frac{1}{p}}\bigg{]}\aod\geq \delta,
\end{align*}
it follows that for each fixed $L>0$ $\limsup_{k\to\infty} \frac{\sum_{p\leq y_k} \frac{X_p}{\sqrt{p}}}{\sqrt{\sum_{p\leq y_k} \frac{1}{p}}}\geq L$, \textit{a.s.}
Similarly, we can conclude that for each fixed $L>0$ $\liminf_{k\to\infty} \frac{\sum_{p\leq y_k} \frac{X_p}{\sqrt{p}}}{\sqrt{\sum_{p\leq y_k} \frac{1}{p}}}\leq -L$, \textit{a.s.} \end{proof}

\begin{proof}[Proof of item ii] Take $\lambda=\lambda(y) = \sqrt{\sum_{p\leq y}\frac{1}{p} }$ in Lemma \ref{lemma aproximando a serie em sigma pela serie em meio } and let $y=\exp((2\sigma-1)^{-1})$. Since $\lim_{y\to\infty}\lambda(y)=\infty$, it follows that there is a subsequence $y_k\to\infty$ for which $\sum_{k=1}^\infty \exp(-d\lambda^2(y_k))<\infty$ and hence, by the Borel-Cantelli Lemma, it \textit{a.s.} holds that
\begin{equation*}
\limsup_{k\to\infty} \frac{\bigg{|} \sum_{p\in\mathcal{P}} \frac{X_p}{p^{\sigma_k}}-\sum_{p\leq y_k} \frac{X_p}{\sqrt{p}}\bigg{|}}{\sqrt{\sum_{p\leq y_k}\frac{1}{p} }} \leq 2,
\end{equation*}
where $y_k=\exp((2\sigma_k-1)^{-1})$. This combined with Lemma \ref{lemma lim sup acima da variancia} gives \textit{a.s.} 
\begin{align*}
\limsup_{\sigma\to1/2^+} \frac{\sum_{p\in\mathcal{P}} \frac{X_p}{p^\sigma}}{\sum_{p\leq y}\frac{1}{p}}&\geq \limsup_{k\to\infty} \frac{ \sum_{p\leq y_k} \frac{X_p}{\sqrt{p}} -\bigg{|}
\sum_{p\in\mathcal{P}} \frac{X_p}{p^{\sigma_k}}-\sum_{p\leq y_k} \frac{X_p}{\sqrt{p}} \bigg{|} }{\sqrt{\sum_{p\leq y_k}\frac{1}{p}}}\\
&\geq \limsup_{k\to\infty}\bigg{(} \frac{ \sum_{p\leq y_k} \frac{X_p}{\sqrt{p}}}{\sqrt{\sum_{p\leq y_k}\frac{1}{p}}} -3\bigg{)}\\
&=\infty.
\end{align*}
Similarly, we conclude that $\liminf_{\sigma\to1/2^+} \sum_{p\in\mathcal{P}} \frac{X_p}{p^\sigma}=-\infty$, \textit{a.s.} Since $F(\sigma)$ is \textit{a.s.} analytic,
it follows that there is an infinite number of $\sigma>1/2$ for which $F(\sigma)=0$. \end{proof}

\subsection{Proof of Theorem \ref{Teorema 1} (ii), the general case}
The following Lemma is an adaptation of \cite{boviergeometricseries}, Theorem 1.2:
\begin{lemma}\label{lemma central do limite} Assume that $\mathcal{P}$ satisfies P1-P2 and that $\sum_{p\in\mathcal{P}}\frac{1}{p}=\infty$. Then
\begin{equation}\label{equacao central do limite 2}
\frac{\sum_{p\in\mathcal{P}}\frac{X_p}{p^\sigma}}{\sqrt{\sum_{p\in\mathcal{P}}\frac{1}{p^{2\sigma}}}}\to_d \mathcal{N}(0,1), \mbox{ as } \sigma\to\frac{1}{2}^+.
\end{equation}
\end{lemma}
\begin{proof} Let $V(\sigma)=\sqrt{\sum_{p\in\mathcal{P}}\frac{1}{p^{2\sigma}}}$. Observe that
$V(\sigma)\to\infty$ as $\sigma\to 1/2^+$: For each fixed $y>0$
\begin{equation*}
\liminf_{\sigma\to 1/2^+}\sum_{p\in\mathcal{P}}\frac{1}{p^{2\sigma}}\geq \lim_{\sigma\to 1/2^+}\sum_{p\leq y}\frac{1}{p^{2\sigma}}=\sum_{p\leq y}\frac{1}{p}.
\end{equation*}
Thus, by making $y\to\infty$ in the equation above, we obtain the desired claim.

For each fixed $\sigma>1/2$, by the Kolmogorov one series Theorem, we have that $\sum_{p\leq y}\frac{X_p}{p^\sigma}$ converges almost surely as $y\to\infty$. Since $(X_p)_{p\in\mathcal{P}}$ are independent, by the dominated convergence theorem:
\begin{align*}
\varphi_\sigma(t)&:=\EE \exp \bigg{(}\frac{it}{V(\sigma)}\sum_{p\in\mathcal{P}}\frac{X_p}{p^\sigma}\bigg{)}=\lim_{y\to\infty}\EE \exp \bigg{(}\frac{it}{V(\sigma)}\sum_{p\leq y}\frac{X_p}{p^\sigma} \bigg{)}\\
&=\prod_{p\in\mathcal{P}}\cos\bigg{(}\frac{t}{V(\sigma)p^\sigma}\bigg{)}.
\end{align*}
We will show that for each fixed $t\in\RR$, $\varphi_\sigma(t)\to \exp(-t^2/2)$ as $\sigma\to 1/2^+$. Observe that $\varphi_\sigma(t)=\varphi_\sigma(-t)$, so we may assume $t\geq 0$. Thus, for each fixed $t\geq 0$
we may choose $\sigma>1/2$ such that $0\leq\frac{t}{V(\sigma)p^\sigma}\leq \frac{1}{100}$ and $0\leq 1-\cos\big{(}\frac{t}{V(\sigma)p^\sigma}\big{)}\leq \frac{1}{100}$, for all $p\in\mathcal{P}$.

For $|x|\leq 1/100$, we have that $\log(1-x)=-x+O(x^2)$ and $\cos(x)=1-\frac{x^2}{2}+O(x^4)$. Further,
$1-\cos(x)=2\sin^2(x/2)\leq \frac{x^2}{2}$. Thus, we have:
\begin{align*}
\log \varphi_\sigma(t)&=\sum_{p\in\mathcal{P}}\log \cos\bigg{(}\frac{t}{V(\sigma)p^\sigma}\bigg{)}\\
&=\sum_{p\in\mathcal{P}}\log\bigg{(}1-\bigg{(}1-\cos\bigg{(}\frac{t}{V(\sigma)p^\sigma}\bigg{)}\bigg{)}\bigg{)}\\
&=-\sum_{p\in\mathcal{P}}\bigg{(}1-\cos\bigg{(}\frac{t}{V(\sigma)p^\sigma}\bigg{)}\bigg{)}+\sum_{p\in\mathcal{P}}O\bigg{(}1-\cos\bigg{(}\frac{t}{V(\sigma)p^\sigma}\bigg{)}\bigg{)}^2 \\
&=-\sum_{p\in\mathcal{P}}\bigg{(}\frac{t^2}{2V^2(\sigma)p^{2\sigma}}+O\bigg{(}\frac{t^4}{V^4(\sigma)p^{4\sigma}} \bigg{)}  \bigg{)}+\sum_{p\in\mathcal{P}}O\bigg{(}\frac{t^4}{V^4(\sigma)p^{4\sigma}} \bigg{)}\\
&=-\frac{t^2}{2V^2(\sigma)}\sum_{p\in\mathcal{P}}\frac{1}{p^{2\sigma}}+\sum_{p\in\mathcal{P}}O\bigg{(}\frac{t^4}{V^4(\sigma)p^{2}} \bigg{)}\\
&=-\frac{t^2}{2}+O\bigg{(}\frac{t^4}{V^4(\sigma)}\bigg{)}.
\end{align*}
We conclude that $\varphi_\sigma(t)\to \exp(-t^2/2)$ as $\sigma\to 1/2^+$.  \end{proof}

\begin{proof}[Proof of item ii] Let $V(\sigma)$ be as in the proof of Lemma \ref{lemma central do limite}.
Since $V(\sigma)\to\infty$ as $\sigma\to 1/2^+$, we have, for each fixed $y>0$
\begin{equation*}
\limsup_{\sigma\to1/2^+}\frac{1}{V(\sigma)}\sum_{p\in\mathcal{P}}\frac{X_p}{p^\sigma}=\limsup_{\sigma\to1/2^+}\frac{1}{V(\sigma)}\sum_{p>y}\frac{X_p}{p^\sigma}.
\end{equation*}
Thus, for each fixed $L>0$,
\begin{equation*}
\limsup_{\sigma\to1/2^+}\frac{1}{V(\sigma)}\sum_{p\in\mathcal{P}}\frac{X_p}{p^\sigma}\geq L
\end{equation*}
is a tail event. By Lemma \ref{lemma central do limite}, $\frac{1}{V(\sigma)}\sum_{p\in\mathcal{P}}\frac{X_p}{p^\sigma}\to_d\mathcal{N}(0,1)$, as $\sigma\to1/2^+$. Thus, this tail event has positive probability (see the proof of Lemma \ref{lemma lim sup acima da variancia}). By the Kolmogorov zero or one Law, \textit{a.s.}:
\begin{equation*}
\limsup_{\sigma\to1/2^+}\frac{1}{V(\sigma)}\sum_{p\in\mathcal{P}}\frac{X_p}{p^\sigma}= \infty.
\end{equation*}
Similarly, \textit{a.s.}:
\begin{equation*}
\liminf_{\sigma\to1/2^+}\frac{1}{V(\sigma)}\sum_{p\in\mathcal{P}}\frac{X_p}{p^\sigma}=-\infty.
\end{equation*}
Since $F(\sigma)=\sum_{p\in\mathcal{P}}\frac{X_p}{p^\sigma}$ is \textit{a.s.} an analytic function, with probability $1$ we have that $F(\sigma)=0$ for an infinite number of $\sigma\to1/2^+$. \end{proof}

\noindent \textbf{Acknowledgments.} The proof of Theorem \ref{Teorema 1} item ii was initially presented
only in the case $\pi(x)\ll \frac{x}{\log x}$. I would like to thank the anonymous referee for pointing that this should not be a necessary condition for the existence of an infinite number of real zeros, and for pointing the reference \cite{boviergeometricseries}, from which I could adapt their Theorem 1.2 for random Dirichlet series (Lemma \ref{lemma central do limite}).

{\small{\sc Departamento de Matem\'atica, Universidade Federal de Minas Gerais, Av. Ant\^onio Carlos, 6627, CEP 31270-901, Belo Horizonte, MG, Brazil.} \\
\textit{Email address:} marco@mat.ufmg.br}

\end{document}